\documentclass[a4paper,11pt]{amsart}

\usepackage{amsmath,amssymb,amsthm,mathtools}
\usepackage{float}
\usepackage{color}
\usepackage[colorlinks=true,
linkcolor=blue,
citecolor=blue,
urlcolor=blue,
filecolor=blue,
menucolor=blue,
runcolor=blue]{hyperref}
\usepackage{enumitem}
\usepackage{diagbox}
\usepackage{makecell}
\allowdisplaybreaks

\numberwithin{equation}{section}
\theoremstyle{plain}
\newtheorem{thm}{Theorem}[section]

\newtheorem{lem}[thm]{Lemma}
\newtheorem{cor}[thm]{Corollary}

\makeatletter

\allowdisplaybreaks[3]

\DeclareMathOperator{\vol}{Vol}

\makeatother

\title[Pinching rigidity theorems for normal scalar curvature]
{Pinching rigidity theorems for normal scalar curvature}

\author[J. Q. Ge]{Jianquan Ge}
\address{$^{}$School of Mathematical Sciences, Laboratory of Mathematics and Complex Systems, Beijing Normal University, Beijing 100875, P. R. CHINA.}
\email{jqge@bnu.edu.cn}

\author[F. G. Li]{Fagui Li}
\address{Frontier Interdisciplinary Domain, Beijing Institute of Technology, Zhuhai, Guangdong 519088, P. R. CHINA.}
\email{lifagui@bitzh.edu.cn}

\author[Y. H. Zhang]{Yunheng Zhang$^{*}$}
\address{$^{*}$School of Mathematical Sciences, Laboratory of Mathematics and Complex Systems, Beijing Normal University, Beijing 100875, P. R. CHINA.}
\email{yunheng@mail.bnu.edu.cn}

\thanks{$^{*}$ the corresponding author.}
\thanks{J. Q. Ge is partially supported by NSFC (No. 12571049) and the Fundamental Research Funds for the Central Universities.}
\thanks{F. G. Li is partially supported by  NSFC (No. 12271040, 12501061) and Research Start up Funding of Beijing Institute of Technology (No. 5640011253301).}
\begin{document}
	\begin{abstract}
	Let $M^n$ be an $n$-dimensional closed minimal submanifold immersed in the unit sphere $\mathbb{S}^{n+m}$. Denote by $S$ and $\rho^{\perp}$ the squared norm of the second fundamental form and the normal scalar curvature of $M^n$, respectively. 
	Let $\{A^{\alpha}\}_{\alpha=n+1}^{n+m}$ be the shape operators of $M^n$   with respect to a local orthonormal normal frame. 
Denote by 
$\lambda_{1}$ 
 the largest eigenvalue of the positive semi-definite symmetric matrix
	 $\mathcal{A}=(\langle A^{\alpha},A^{\beta}\rangle)_{m\times m}$.
	 We show that if $\lambda_{1}\leqslant n$ and $\rho^{\perp}\leqslant
\left[{\sqrt{2}n(n-1)}\right]^{-1}
	 \mathop{\inf}\limits_{p\in M}(n-\lambda_{1})(p)$, then 
	  $\rho^{\perp}\equiv 0$, which means the normal bundle of $M^n$ is flat, and further we give the classification of $M^n$.   
	\end{abstract}
	
	\maketitle
	
	\section{Introduction}
	The study of pinching phenomena is a significant topic in differential geometry. In 1968, Simons \cite{Sim} proved that  \emph{if $M^n$ is an $n$-dimensional closed minimal submanifold in the unit sphere $\mathbb{S}^{n+m}$, whose squared norm of the second fundamental form satisfies $0\leqslant S\leqslant \frac{n}{2-\frac{1}{m}}$, then $S\equiv 0$ and $M$ is the great sphere, or $S\equiv \frac{n}{2-\frac{1}{m}}$.} Subsequently, Chern-do Carmo-Kobayashi \cite{CDK} classified the submanifold on which $S\equiv \frac{n}{2-\frac{1}{m}}$ to be only the Clifford minimal hypersurfaces and the Veronese surface in $\mathbb{S}^4$. Notably, Magliaro et al. \cite{MMRS} extended this rigidity result to a complete minimal submanifold. Afterwards, Li-Li \cite{1992 LiLi} and Chen-Xu \cite{CX93} improved Simons' pinching constant to $\frac{2n}{3}$ for the case $m\geqslant 2$. In particular, for $m=1$, the celebrated Chern conjecture (cf. \cite{Chang 1993,Deng Gu and  Wei  2017,Qi Ding  Y.L. Xin. 2011,Li Lei Hongwei Xu  Zhiyuan Xu 2017,Peng and Terng1 1983,Cheng Qing Ming and H Yang 1998}, etc.) asserts that \emph{a closed minimal hypersurface $M^n$ with constant scalar curvature in $\mathbb{S}^{n+1}$ is isoparametric.}
	A recent important development by Tang-Wei-Yan \cite{T-W-Y} and Tang-Yan \cite{T-Y} showed that \emph{$M^n$ is isoparametric if it has constant $r$-th mean curvature $(r=1,...,n-1) $ and nonnegative scalar curvature,} which generalized the result of Almeida and Brito \cite{A-B} for $n=3$ and strongly supported the Chern conjecture.
	Please see the excellent and detailed surveys on this topic 
	by   Scherfner-Weiss-Yau \cite{M Scherfner S Weiss and  Yau 2012} and Ge-Tang \cite{Ge Tang 2012}, etc.
	
  The renowned DDVV conjecture \cite{DDVV}  was independently proved by Ge-Tang \cite{GT} and by Lu \cite{Lu}. Using a generalized DDVV inequality involving commutators of shape operators, Lu \cite{Lu} 
	proved that \emph{for a closed, minimally immersed submanifold $M^n$ in the unit sphere
		$\mathbb{S}^{n+m}$, if $0\leqslant S+\lambda_{2}\leqslant n$, then $M^n$ is totally geodesic, or is a Clifford torus in $\mathbb{S}^{n+1}\subset \mathbb{S}^{n+m}$, or is a Veronese surface in $\mathbb{S}^{4}\subset \mathbb{S}^{2+m}$, where $\lambda_{2}$ is the second largest eigenvalue of the positive semi-definite symmetric matrix $\mathcal{A}=(\langle A^{\alpha},A^{\beta}\rangle)_{m\times m}$ and $\left\lbrace A^{\alpha}\right\rbrace _{\alpha=n+1}^{n+m}$ are the shape operators of $M^n$ with respect to a given normal orthonormal frame.} The above pinching theorem generalized the earlier pinching theorems established by Simons \cite{Sim}, Chern, do Carmo, and Kobayashi \cite{CDK}, Yau \cite{Yau 1974,Yau 1975}, Shen \cite{YBShen 1989}, as well as Li and Li \cite{1992 LiLi}, etc. 
		For further details on the DDVV-type inequalities, we refer the reader to \cite{GLTZ,GLZ}, etc.
		
	Recently, Ge-Tao-Zhou \cite{GTZ} studied a class of austere submanifold and showed that  \emph{$M$ is either a totally geodesic submanifold or the Clifford torus $\mathbb{S}^{p}(\sqrt{\frac{1}{2}})\times \mathbb{S}^{p}(\sqrt{\frac{1}{2}})$ with $n=2p$ if $S\leqslant n$.} Moreover, for compact submanifolds in Euclidean and spherical space forms satisfying a Ricci curvature pinching condition, Ge-Tao-Zhou \cite{GTZ2} established that \emph{such a submanifold is either homeomorphic to the n-sphere or isometric to the Clifford torus.}
	Under some slight inequality conditions involving the normal scalar curvature, Ding-Ge-Li-Yang \cite{DGLY} also obtained some pinching theorems of minimal surfaces. In particular, they proved that \emph{for a closed minimal surface $M^{2}$ immersed in the unit sphere $\mathbb{S}^{2+m}$, if the normal bundle of $M$ is flat and $S+\lambda_{2}>2$, then $\mathop{\max}\limits_{p\in M}(S+\lambda_{2})(p)>\frac{20}{9}$.}
	
	In this paper, for the sake of clarity, we let $\lambda_{1}\geqslant\lambda_{2}\geqslant\cdots\geqslant\lambda_{m}$ be the eigenvalues of the positive semi-definite symmetric matrix $\mathcal{A}=(\langle A^{\alpha},A^{\beta}\rangle)_{m\times m}$.  The normal scalar curvature $\rho^{\perp}$
	is defined by
	\[\rho^{\perp}=\frac{1}{n(n-1)}|R^\perp|=\frac{2}{n(n-1)}\left(\sum_{1=i<j}^{n}\sum_{1=r<s}^{m}\left\langle R^{\perp}(e_{i},e_{j})\xi_{r},\xi_{s}\right\rangle^2\right)^{\frac{1}{2}},\]
	where $\{e_1,\ldots, e_n\}$ is an orthonormal basis of the tangent space, $\{\xi_{1},\ldots,\xi_{m}\}$ is an orthonormal basis of the normal space, and $R^\perp$ is the curvature tensor of the normal bundle.
	The following new pinching rigidity theorems and classification results of $\lambda_{1}$  and $\rho^{\perp}$  can be obtained. 
\begin{thm}\label{thm S0n p=0}
	Let $M^n$ be an $n$-dimensional closed minimal submanifold immersed in the unit sphere $\mathbb{S}^{n+m}$. If $\lambda_{1}\leqslant n$ and $\rho^{\perp}\leqslant\frac{1}{\sqrt{2}n(n-1)}\mathop{\inf}\limits_{p\in M}(n-\lambda_{1})(p)$, then  $M^n$ must be one of the following:
	\begin{enumerate}
	\item[\rm (1)]   the great sphere with $S\equiv 0$;
		\item[\rm (2)]  the product $\prod\limits_{i=1}^{r+1}\mathbb{S}^{n_{i}}(\sqrt{\frac{n_{i}}{n}})$, where $\sum\limits_{i=1}^{r+1}n_{i}=n$, $1\leqslant r\leqslant m$. Moreover, $M^n$ satisfies the following properties:  (i) $\rho^{\perp}=0$; (ii) $\lambda_{1}=n$; (iii) $S=rn$.
		\end{enumerate}
\end{thm}

By Theorem \ref{thm S0n p=0}, we have the following theorem.
  \begin{thm}\label{spe}
	Let $M^n$ be an $n$-dimensional closed minimal submanifold immersed in the unit sphere $\mathbb{S}^{n+m}$. If $ S\leqslant n$ and $\rho^{\perp}\leqslant\frac{1}{\sqrt{2}n(n-1)}\mathop{\inf}\limits_{p\in M}(n-\lambda_{1})(p)$, then $M^n$ must be one of the following:
\begin{enumerate}
\item[\rm (1)]   the great sphere with $S\equiv 0$; 
\item[\rm (2)]  the Clifford torus $\mathbb{S}^{k}\Big(\sqrt\frac{k}{n}\Big)\times \mathbb{S}^{n-k}\Big(\sqrt\frac{n-k}{n}\Big)$ with $S\equiv n$ and $1\leqslant k\leqslant n-1$. 
\end{enumerate}
  \end{thm}
  In particular, in the case that $\rho^{\perp}$ is constant, we prove the following corollaries:
  \begin{cor}\label{constant}
  	Let $M^n$ be an $n$-dimensional closed minimal submanifold immersed in the unit sphere $\mathbb{S}^{n+m}$ with constant normal scalar curvature $\rho^{\perp}$. If $$\delta=\left(\int_{M}(3S-n) dM\right)^2\\-4\vol(M)\int_{M}S(S-n)dM\geqslant 0, $$ then
    one of the following two cases must occur:
  		\begin{enumerate}
  			\item[\rm (1)]  	$\rho^{\perp}\leqslant\frac{\int_{M}(n-3S)dM-\sqrt{\delta}}{2n(n-1)\vol(M)}$;
  			\item[\rm (2)]   $\rho^{\perp}\geqslant\frac{\int_{M}(n-3S)dM+\sqrt{\delta}}{2n(n-1)\vol(M)}.$
  		\end{enumerate}
  	\end{cor}

\begin{cor}\label{Sconstant}
   Let $M^n$ be an $n$-dimensional closed minimal submanifold immersed in the unit sphere $\mathbb{S}^{n+m}$ with constant $S$ and constant normal scalar curvature $\rho^{\perp}$. If $\rho^{\perp}\leqslant \frac{n-3S+\sqrt{5S^2-2nS+n^2}}{2n(n-1)}$, then $M^n$ must be either the great sphere or the Clifford torus in Theorem \ref{spe}.
  	\end{cor}
	\section{Preliminaries}
We consider an $n$-dimensional closed minimal submanifold $M^n$ immersed in the unit sphere $\mathbb{S}^{n+m}$. Let $\{e_1,\ldots,e_{n+m}\}$ be a local orthonormal frame on $T(\mathbb{S}^{n+m})$ such that, when restricted to $M^n$, the vectors $\{e_1,\ldots,e_n\}$  lie in the tangent bundle $T(M)$ and $\{e_{n+1},\ldots,e_{n+m}\}$ lie in the normal bundle $T^{\perp}(M)$, respectively. We shall make use of the following convention on the ranges of the indices:
\begin{equation*}
	1\leqslant i,j,k,\ldots\leqslant n,\quad n+1\leqslant\alpha,\beta,\gamma,\ldots\leqslant n+m,\quad 1\leqslant A,B,C,\ldots\leqslant n+m.
\end{equation*}
Denote by ($\omega_{A}$) the
metric 1-forms and by ($\omega_{AB}$) the connection 1-forms associated with the frame $\{e_1,\ldots,e_{n+m}\}$.
 Define the matrices $A^{\alpha}=(h_{ij}^{\alpha})_{n\times n}$ via \[\omega_{i\alpha}=\sum\limits_{j}h_{ij}^{\alpha}\omega_{j}.\] The second
 fundamental form of $M$ is then given by \[h=\sum\limits_{i,j,\alpha}h_{ij}^{\alpha}\omega_{i}\omega_{j}e_{\alpha}.\]  
 We shall sometimes refer to the components $h_{ij}^{\alpha}$ as the second fundamental form. Moreover, we define  
 \begin{equation*}
 	S:=|h|^2,\qquad
 	\rho^{\perp}_0:=\sum\limits_{\alpha,\beta}|[A^{\alpha},A^{\beta}]|^2.
 \end{equation*}
Therefore, the relation between $\rho^{\perp}_0$ and $\rho^{\perp}$ is
\begin{equation}\label{rel}
 \rho^{\perp}=\frac{1}{n(n-1)}\sqrt{\rho^{\perp}_0}.
\end{equation}
The covariant derivatives $h_{ijk}^{\alpha}$, $h_{ijkl}^{\alpha}$ are defined by
\begin{equation*}
	\begin{aligned}
&\sum_{k}h_{ijk}^{\alpha}\omega_{k}=dh_{ij}^\alpha+\sum_{l}h_{lj}^{\alpha}\omega_{li}+\sum_{l}h_{il}^{\alpha}\omega_{lj}+\sum_{\beta}h_{ij}^{\beta}\omega_{\beta\alpha},\\
&\sum_{l}h_{ijkl}^{\alpha}\omega_{l}=dh_{ijk}^{\alpha}+\sum_{m}h_{mjk}^{\alpha}\omega_{mi}+\sum_{m}h_{imk}^{\alpha}\omega_{mj}+\sum_{m}h_{ijm}^{\alpha}\omega_{mk}+\sum_{\beta}h_{ijk}^{\beta}\omega_{\beta\alpha}.		
	\end{aligned}
\end{equation*}
The Codazzi equation and Ricci formula are
\begin{equation*}
	\begin{aligned}
		&h_{ijk}^{\alpha}=h_{ikj}^{\alpha},\\
		&h_{ijkl}^{\alpha}-h_{ijlk}^{\alpha}=\sum_{p}h_{ip}^{\alpha}R_{pjkl}+\sum_{p}h_{pj}^{\alpha}R_{pikl}-\sum_{\beta}h_{ij}^{\beta}R_{\alpha\beta kl},
	\end{aligned}
\end{equation*}
where
\begin{equation*}
	\begin{aligned}
		&R_{ijkl}=\delta_{ik}\delta_{jl}-\delta_{il}\delta_{jk}+\sum_{\alpha}(h_{ik}^{\alpha}h_{jl}^{\alpha}-h_{il}^{\alpha}h_{jk}^{\alpha}),\\
		&R_{\alpha\beta kl}=\sum_{i}(h_{ik}^{\alpha}h_{il}^{\beta}-h_{il}^{\alpha}h_{ik}^{\beta}).
	\end{aligned}
\end{equation*}
For convenience, we define $|\nabla h|^2=\sum\limits_{\alpha,i,j,k}(h_{ijk}^{\alpha})^2$.
The Laplacian $\Delta h_{ij}^{\alpha}$ of the second fundamental form $h_{ij}^{\alpha}$ is defined by
\[\Delta h_{ij}^{\alpha}=\sum_{k}h_{ijkk}^{\alpha}.\]
The matrix inner product is defined as
\[\langle A,B\rangle={\rm Tr}(AB^{\top}),\]where $B^{\top}$ denotes the transpose of $B$.
For further estimates, we recall the following well-known inequalities and a celebrated pinching theorem due to Li-Li \cite{1992 LiLi} and Chen-Xu \cite{CX93}.

\begin{thm}[{DDVV inequality} \cite{GT, Lu}]
	Let $B_{1},\ldots,B_{m}$ be $n\times n$ real symmetric matrices. Then
	\begin{equation}\label{DDVV}
		\sum_{r,s=1}^{m}\|[B_{r},B_{s}]\|^2\leqslant(\sum_{r=1}^{m}\|B_{r}\|^2)^2,
	\end{equation} 
where $\|\cdot\|^2$ denotes the sum of squares of the entries of a matrix and $[A,B]:=AB-BA$ is the commutator of the matrices $A$ and $B$. 
\end{thm}
\begin{thm}[Böttcher-Wenzel inequality \cite{Lu}] 
\textit{Let $X$, $Y$ be two $n\times n$ complex matrices. Then
\begin{equation}\label{BW}
	\|[X,Y]\|^2\leqslant 2\|X\|^2\|Y\|^2.
\end{equation}}
\end{thm}
\begin{thm}[\cite{1992 LiLi,CX93}]\label{LiLi}
Let $M^n$ be an $n$-dimensional closed minimal submanifold immersed in the unit sphere $\mathbb{S}^{n+m}$ with $m\geqslant2$.
If $0\leqslant S\leqslant \frac{2}{3}n$ everywhere on $M$,
then $M$ is either a totally geodesic submanifold or a Veronese surface in $\mathbb{S}^{2+m}$.
\end{thm} 
\begin{cor}\label{LL}
Let $M^n$ be an $n$-dimensional closed minimal submanifold immersed in the unit sphere $\mathbb{S}^{n+m}$ with $m\geqslant 1$. If $S>0$, then $\mathop{\max}\limits_{p\in M} S(p)\geqslant \frac{2}{3}n$.
\end{cor}
\begin{proof}
	When $m=1$, Simons' inequality \cite{Sim} gives \[\int_{M}S(S-n)dM\geqslant0.\]If $S>0$, then \[\mathop{\max}\limits_{p\in M} S(p)\geqslant n>\frac{2}{3}n.\]
	When $m\geqslant 2$, the result follows directly from Theorem \ref{LiLi}.
\end{proof}
	\section{Proof of the main results}
 We first establish several lemmas that will be used in the subsequent estimates.
\begin{lem}\label{Delta}
	Let $M^n$ be an n-dimensional closed minimal submanifold immersed in the unit sphere $\mathbb{S}^{n+m}$. Then we have
	\begin{equation*}
		\begin{aligned}
			\frac{1}{4}\Delta\rho^{\perp}_0=&n\rho^{\perp}_0+\sum_{\alpha,\beta,\gamma}\langle A^{\alpha}, A^{\gamma}\rangle \langle[A^{\beta},A^{\gamma}],[A^{\alpha},A^{\beta}]\rangle-\sum_{\alpha}|\sum_{\gamma}[[A^{\alpha},A^{\gamma}],A^{\gamma}]|^2\\
			&+\sum_{\alpha,\beta}\langle\sum_{k}[\nabla_{e_{k}} A^{\alpha},\nabla_{e_{k}} A^{\beta}],[A^{\alpha},A^{\beta}]\rangle+\frac{1}{2}\sum_{\alpha,\beta}|\nabla [A^{\alpha},A^{\beta}]|^2.
		\end{aligned}
	\end{equation*}
\end{lem}
\begin{proof}
	Use the Einstein summation convention, a direct computation yields
	\begin{equation*}
		\begin{aligned}
			\Delta h_{ij}^{\alpha}=&nh_{ij}^{\alpha}+2h_{kp}^{\alpha}h_{pj}^{\beta}h_{ik}^{\beta}-h_{kp}^{\alpha}h_{pk}^{\beta}h_{ij}^{\beta}-h_{pi}^{\alpha}h_{pk}^{\beta}h_{jk}^{\beta}-h_{ki}^{\beta}h_{pj}^{\alpha}h_{pk}^{\beta}.
		\end{aligned}
	\end{equation*}
	Then in terms of matrix, Lu \cite{Lu} gives
	\begin{equation*}
		\Delta A^{\alpha}=nA^{\alpha}-\langle A^{\alpha},A^{\beta}\rangle A^{\beta}-[A^{\beta},[A^{\beta},A^{\alpha}]].
	\end{equation*}
	Therefore
	\begin{equation*}
		\nabla [A^{\alpha},A^{\beta}]=\nabla A^{\alpha}A^{\beta}+A^{\alpha}\nabla A^{\beta}-\nabla A^{\beta}A^{\alpha}-A^{\beta}\nabla A^{\alpha},
	\end{equation*}
and
	\begin{equation*}
		\begin{aligned}
			\Delta [A^{\alpha},A^{\beta}]=&\Delta A^{\alpha}A^{\beta}+A^{\alpha}\Delta A^{\beta}+2\sum_{k}\nabla_{e_{k}} A^{\alpha}\nabla_{e_{k}} A^{\beta}\\
			&-\Delta A^{\beta}A^{\alpha}-A^{\beta}\Delta A^{\alpha}-2\sum_{k}\nabla_{e_{k}} A^{\beta}\nabla_{e_{k}} A^{\alpha}\\
			=&[\Delta A^{\alpha},A^{\beta}]+[A^{\alpha},\Delta A^{\beta}]+2\sum_{k}[\nabla_{e_{k}} A^{\alpha},\nabla_{e_{k}} A^{\beta}].
		\end{aligned}
	\end{equation*}
	By a direct computation, we obtain
	\begin{equation}\label{Stype}
		\begin{aligned}
			\frac{1}{2}\Delta\sum_{\alpha,\beta}|[A^{\alpha},A^{\beta}]|^2=&\frac{1}{2}\Delta\sum_{\alpha,\beta} \langle[A^{\alpha},A^{\beta}],[A^{\alpha},A^{\beta}] \rangle\\
			=&\sum_{\alpha,\beta}\langle\Delta [A^{\alpha},A^{\beta}],[A^{\alpha},A^{\beta}] \rangle+\sum_{\alpha,\beta}|\nabla[A^{\alpha},A^{\beta}]|^2.
		\end{aligned}
	\end{equation}
Since
	\begin{equation}\label{sub}
	\begin{aligned}
		&\sum_{\alpha,\beta}\langle\Delta [A^{\alpha},A^{\beta}],[A^{\alpha},A^{\beta}] \rangle\\
			=&2\sum_{\alpha,\beta}\langle[\Delta A^{\alpha},A^{\beta}],[A^{\alpha},A^{\beta}]\rangle+2\sum_{\alpha,\beta}\langle\sum_{k}[\nabla_{e_{k}} A^{\alpha},\nabla_{e_{k}} A^{\beta}],[A^{\alpha},A^{\beta}] \rangle\\
			=&2n\sum_{\alpha,\beta}|[A^{\alpha},A^{\beta}]|^2+2\sum_{\alpha,\beta,\gamma}\langle A^{\alpha}, A^{\gamma}\rangle\langle[A^{\beta},A^{\gamma}],[A^{\alpha},A^{\beta}]\rangle\\
			&+2\sum_{\alpha,\beta,\gamma}\langle[A^{\beta},[A^{\gamma},[A^{\gamma},A^{\alpha}]]],[A^{\alpha},A^{\beta}]\rangle+2\sum_{\alpha,\beta}\langle\sum_{k}[\nabla_{e_{k}} A^{\alpha},\nabla_{e_{k}} A^{\beta}],[A^{\alpha},A^{\beta}] \rangle\\
			=&2n\sum_{\alpha,\beta}|[A^{\alpha},A^{\beta}]|^2+2\sum_{\alpha,\beta,\gamma}\langle A^{\alpha}, A^{\gamma}\rangle\langle[A^{\beta},A^{\gamma}],[A^{\alpha},A^{\beta}]\rangle\\
			&-2\sum_{\alpha}|\sum_{\gamma}[[A^{\alpha},A^{\gamma}],A^{\gamma}]|^2+2\sum_{\alpha,\beta}\langle\sum_{k}[\nabla_{e_{k}} A^{\alpha},\nabla_{e_{k}} A^{\beta}],[A^{\alpha},A^{\beta}] \rangle,
		\end{aligned}
	\end{equation}
 substituting (\ref{sub}) into (\ref{Stype}), this completes the proof of Lemma \ref{Delta}.
	
\end{proof}

\begin{lem}\label{ind}
	Using notations above, we obtain
	\begin{equation}\label{frame}
	-\sum_{\alpha,\beta,\gamma}\langle A^{\alpha}, A^{\gamma}\rangle \langle[A^{\beta},A^{\gamma}],[A^{\alpha},A^{\beta}]\rangle\leqslant\lambda_{1}\rho_{0}^{\perp}.
	\end{equation}
\end{lem}
\begin{proof}
	It is easy to see that the left side of (\ref{frame}) is independent of the selection of the frame. Then at each fixed point, we choose a suitable orthonormal normal frame, such that $\langle A^{\alpha},A^{\beta}\rangle = 0$  for any $n+1\leqslant\alpha\neq\beta\leqslant n+m$. Then we obtain
	\begin{equation*}
		\begin{aligned}
		-\sum_{\alpha,\beta,\gamma}\langle A^{\alpha}, A^{\gamma}\rangle \langle[A^{\beta},A^{\gamma}],[A^{\alpha},A^{\beta}]\rangle=&-\sum_{\alpha,\beta}\|A^{\alpha}\|^2 \langle[A^{\beta},A^{\alpha}],[A^{\alpha},A^{\beta}]\rangle\\
		=&\sum_{\alpha,\beta}\|A^{\alpha}\|^2|[A^{\alpha},A^{\beta}]|^2\\
		\leqslant&\lambda_{1}\rho_{0}^{\perp}.
	\end{aligned}
	\end{equation*}
\end{proof}
\begin{lem}\label{DV}
	Using notations above, we obtain
	\begin{equation*}
	-\sum_{\alpha,\beta}\langle\sum_{k}[\nabla_{e_{k}} A^{\alpha},\nabla_{e_{k}} A^{\beta}],[A^{\alpha},A^{\beta}]\rangle\leqslant|\nabla h|^2\sqrt{\rho^{\perp}_0}.
	\end{equation*}
\end{lem}

\begin{proof}
	Using the Cauchy inequality, we have
	\begin{equation}\label{equation ekAABBA upper bound}
		\begin{aligned}
			-\sum_{\alpha,\beta}\langle\sum_{k}[\nabla_{e_{k}} A^{\alpha},\nabla_{e_{k}} A^{\beta}],[A^{\alpha},A^{\beta}]\rangle
			\leqslant&\big|\sum_{\alpha,\beta}\langle\sum_{k}[\nabla_{e_{k}} A^{\alpha},\nabla_{e_{k}} A^{\beta}],[A^{\alpha},A^{\beta}]\rangle\big|\\
			\leqslant&\sum_{\alpha,\beta}\big|\sum_{k}[\nabla_{e_{k}} A^{\alpha},\nabla_{e_{k}} A^{\beta}]\big|\big|[A^{\alpha},A^{\beta}]\big|\\
			\leqslant&\sqrt{\sum_{\alpha,\beta}\big|\sum_{k}[\nabla_{e_{k}} A^{\alpha},\nabla_{e_{k}} A^{\beta}]\big|^2}\sqrt{\sum_{\alpha,\beta}\big|[A^{\alpha},A^{\beta}]\big|^2}\\
			=&\sqrt{\sum_{\alpha,\beta}\sum_{k,l}\langle[\nabla_{e_{k}} A^{\alpha},\nabla_{e_{k}} A^{\beta}],[\nabla_{e_{l}} A^{\alpha},\nabla_{e_{l}} A^{\beta}]\rangle}\sqrt{\rho^{\perp}_0}.
		\end{aligned}			
	\end{equation}
	At each fixed point, one may choose a local orthonormal frame $\{e_{i}\}_{i=1}^{n}$ such that
\begin{equation*}
	\sum_{\alpha,\beta}\sum_{k,l}\langle[\nabla_{e_{k}} A^{\alpha},\nabla_{e_{k}} A^{\beta}],[\nabla_{e_{l}} A^{\alpha},\nabla_{e_{l}} A^{\beta}]\rangle=\sum_{\alpha,\beta}\sum_{k}\big|[\nabla_{e_{k}} A^{\alpha},\nabla_{e_{k}} A^{\beta}]\big|^2.
\end{equation*}
By the DDVV inequality (\ref{DDVV}), we obtain
\begin{equation*}
	\begin{aligned}
	\sum_{\alpha,\beta}\sum_{k}\big|[\nabla_{e_{k}} A^{\alpha},\nabla_{e_{k}} A^{\beta}]\big|^2\leqslant&\sum_{k}\left(\sum_{\alpha}|\nabla_{e_{k}} A^{\alpha}|^2\right)^2\\
	=&\sum_{k}\sum_{\alpha,\beta}|\nabla_{e_{k}} A^{\alpha}|^2|\nabla_{e_{k}} A^{\beta}|^2\\
	\leqslant&|\nabla h|^4.
	\end{aligned}
\end{equation*}
Hence
\begin{equation*}
	\sqrt{\sum_{\alpha,\beta}\sum_{k,l}\langle[\nabla_{e_{k}} A^{\alpha},\nabla_{e_{k}} A^{\beta}],[\nabla_{e_{l}} A^{\alpha},\nabla_{e_{l}} A^{\beta}]\rangle}\leqslant |\nabla h|^2,
\end{equation*}
which completes the proof 
by (\ref{equation ekAABBA upper bound}).
\end{proof}

\begin{lem}\label{bd}
	Using notations above, we obtain
	\begin{equation*}
	\sum_{\alpha}\big|\sum_{\gamma}[[A^{\alpha},A^{\gamma}],A^{\gamma}]\big|^2\leqslant2S\rho^{\perp}_0.
	\end{equation*}
\end{lem}
\begin{proof}
		Using the Cauchy inequality and Böttcher-Wenzel inequality (\ref{BW}), we get
	\begin{equation*}
		\begin{aligned}
		\sum_{\alpha}\big|\sum_{\gamma}[[A^{\alpha},A^{\gamma}],A^{\gamma}]\big|^2&\leqslant\sum_{\alpha}\big(\sum_{\gamma}|[[A^{\alpha},A^{\gamma}],A^{\gamma}]|\big)^2\\
		&\leqslant\sum_{\alpha}\big(\sum_{\gamma}\sqrt{2}|[A^{\alpha},A^{\gamma}]|\cdot|A^{\gamma}|\big)^2\\
		&\leqslant\sum_{\alpha}2\sum_{\gamma}|[A^{\alpha},A^{\gamma}]|^2\sum_{\gamma}|A^{\gamma}|^2\\
		&=2S\rho^{\perp}_0.
		\end{aligned}
	\end{equation*}
\end{proof}
By Lemmas \ref{Delta}, \ref{ind}, \ref{DV} and \ref{bd}, one has the following theorem.
	\begin{thm}\label{main}
	Let $M^n$ be an $n$-dimensional closed minimal submanifold immersed in the unit sphere $\mathbb{S}^{n+m}$. If $\lambda_{1}\leqslant n$ and $\rho^{\perp}\leqslant\frac{1}{\sqrt{2}n(n-1)}\mathop{\inf}\limits_{p\in M}(n-\lambda_{1})(p)$, then $\rho^{\perp}\equiv 0$, which means the normal bundle of $M^n$ is flat.
	\end{thm}
\begin{proof}
From Lemmas \ref{Delta}, \ref{ind}, \ref{DV} and \ref{bd}, we obtain
\begin{equation}\label{fr}
	\begin{aligned}
	\frac{1}{4}\Delta\rho^{\perp}_0\geqslant&n\rho^{\perp}_0-\lambda_{1}\rho_{0}^{\perp}-2S\rho_{0}^{\perp}-|\nabla h|^2\sqrt{\rho_{0}^{\perp}}+\frac{1}{2}\sum_{\alpha,\beta}|\nabla [A^{\alpha},A^{\beta}]|^2.
	\end{aligned}
\end{equation}
Integrating (\ref{fr}) gives
\begin{equation}\label{int}
	\begin{aligned}
	0\leqslant\int_{M}&\frac{1}{2}\sum_{\alpha,\beta}|\nabla [A^{\alpha},A^{\beta}]|^2dM\\
	=\int_{M}&-n\rho^{\perp}_0+\lambda_{1}\rho_{0}^{\perp}+2S\rho_{0}^{\perp}+|\nabla h|^2\sqrt{\rho_{0}^{\perp}} dM\\
	=\int_{M}&\sqrt{\rho^{\perp}_0}\left(|\nabla h|^2+\big(\lambda_{1}+2S-n\big)\sqrt{\rho^{\perp}_0}\right)dM.
\end{aligned}
\end{equation}
On the other hand, from Simons-type formula
\begin{equation}\label{Simons}
	\begin{aligned}
		\frac{1}{2}\Delta S=|\nabla h|^2+nS-\sum_{\alpha,\beta}|[A^{\alpha},A^{\beta}]|^2-\sum_{\alpha,\beta}|\langle A^{\alpha},A^{\beta}\rangle|^2,
	\end{aligned}
\end{equation}
since $\mathcal{A}=(\langle A^{\alpha},A^{\beta}\rangle)_{m\times m}$, integration of (\ref{Simons}) leads to
\begin{equation}\label{nablah}
	\int_{M}|\nabla h|^2dM=\int_{M}\rho^{\perp}_0-nS+\|\mathcal{A}\|^2 dM.
\end{equation}
Denote $C=\mathop{\max}\limits_{p\in M}\rho^{\perp}_0(p)$. If $C=0$, then $\rho_{0}^{\perp}\equiv 0$. Hence, we assume $C>0$ in what follows.
Since $$\|\mathcal{A}\|^2=\sum_{i}\lambda_{i}^2\leqslant \lambda_{1}\sum_{i}\lambda_{i}=\lambda_{1}S,$$ then (\ref{int}) becomes
\begin{equation}\label{key}
	\begin{aligned}
	0\leqslant\int_{M}&\sqrt{\rho^{\perp}_0}|\nabla h|^2+(\lambda_{1}+2S-n)\rho^{\perp}_0dM\\
	\leqslant\int_{M}&\sqrt{C}(\rho^{\perp}_0-nS+\|\mathcal{A}\|^2)+(\lambda_{1}+2S-n)\rho^{\perp}_0dM\\
	\leqslant\int_{M}&\sqrt{C}(\rho^{\perp}_0-nS+S\lambda_{1})+(\lambda_{1}+2S-n)\rho^{\perp}_0dM\\
	=\int_{M}&\rho^{\perp}_0(\sqrt{C}-n+\lambda_{1})+\sqrt{C}S(-n+\lambda_{1}+\frac{2\rho^{\perp}_0}{\sqrt{C}})dM.
		\end{aligned}
\end{equation}
Combining (\ref{rel}) with $\lambda_{1}\leqslant n$, $\rho^{\perp}\leqslant\frac{1}{\sqrt{2}n(n-1)}\mathop{\inf}\limits_{p\in M}(n-\lambda_{1})(p)$, we have 
\begin{equation}\label{equation p0C2n}
 \rho^{\perp}_0\leqslant C\leqslant\dfrac{1}{2}\mathop{\inf}\limits_{p\in M}(n-\lambda_{1})^2(p).
\end{equation}
Therefore, (\ref{key}) becomes
\begin{equation}\label{leq}
	\begin{aligned}
	0\leqslant\int_{M}&\rho^{\perp}_0(\sqrt{C}-n+\lambda_{1})+\sqrt{C}S\left( -n+\lambda_{1}+\frac{\mathop{\inf}\limits_{p\in M}(n-\lambda_{1})^2}{\sqrt{C}}\right) dM\\
	\leqslant\int_{M}&\rho^{\perp}_0(\sqrt{C}-n+\lambda_{1})+\sqrt{C}S(n-\lambda_{1})\left(\frac{\mathop{\inf}\limits_{p\in M}(n-\lambda_{1})}{\sqrt{C}}-1\right)dM.
	\end{aligned}
\end{equation}
The equalities in (\ref{leq}) hold if and only if $\rho_{0}^{\perp}=\frac{1}{2}\mathop{\inf}\limits_{p\in M}(n-\lambda_{1})^2(p)={\rm Constant}$.
Furthermore, from (\ref{nablah}) we obtain
\begin{equation}\label{conv}
	\int_{M}\rho^{\perp}_0dM
	\geqslant \int_{M}nS-\|\mathcal{A}\|^2dM\geqslant\int_{M}S(n-\lambda_{1})dM.
\end{equation}
By (\ref{equation p0C2n}), $\mathop{\inf}\limits_{p\in M}(n-\lambda_{1})(p)\geqslant \sqrt{C}$, thus substituting (\ref{conv}) into (\ref{leq}) implies
\begin{equation*}
	\begin{aligned}
	0\leqslant&\int_{M}\rho^{\perp}_0(\sqrt{C}-n+\lambda_{1})+\rho^{\perp}_0\left(\mathop{\inf}\limits_{p\in M}(n-\lambda_{1})-\sqrt{C}\right)dM\\
	=&\int_{M}\rho^{\perp}_0\left(\mathop{\inf}\limits_{p\in M}(n-\lambda_{1})-(n-\lambda_{1})\right)dM\leqslant0.
	\end{aligned}
\end{equation*}
Hence, the equalities hold for all the above inequalities simultaneously, which yields that either $\rho^{\perp}_0\equiv 0$, or $\rho^{\perp}_0\equiv\dfrac{1}{2}(n-\lambda_{1})^2={\rm Constant}\neq0$ and there exists a positive integer $r$ such that $\lambda_{1}=\lambda_{2}=\cdots=\lambda_{r}$, $\lambda_{r+1}=\cdots=\lambda_{m}=0$ . 
In the latter case, one has $\lambda_{1}$ is  constant, then $S=r\lambda_{1}$ is also constant.
By Corollary \ref{LL},
if $S>0$, then  $$\max\limits_{p\in M} S(p)=S\geqslant \dfrac{2n}{3}.$$ Combining this with the assumption condition, we derive that
\begin{equation*}
	\frac{2n}{3r}\leqslant \lambda_{1}\leqslant n.
\end{equation*}
Substituting $\rho^{\perp}_0\equiv\dfrac{1}{2}(n-\lambda_{1})^2$ into (\ref{nablah}) gives
\begin{equation}\label{end}
	\begin{aligned}
	0\leqslant\int_{M}|\nabla h|^2dM=&\int_{M}\frac{1}{2}(n-\lambda_{1})^2-nr\lambda_{1}+r\lambda_{1}^2 dM\\
	=&\int_{M}( \frac{1}{2}+r) \lambda_{1}^2-n(1+r)\lambda_{1}+\frac{1}{2}n^2dM\\
	=&\int_{M}( \frac{1}{2}+r)(\lambda_{1}-n)(\lambda_{1}-\frac{n}{1+2r})dM.\\
	\end{aligned}
\end{equation}
Since the right-hand side of (\ref{end}) is nonpositive, we must have $\lambda_{1}\equiv n$ and $\rho^{\perp}_0\equiv 0$, which contradicts the assumption $\rho^{\perp}_0\not\equiv 0$ (or $C>0$) .
Thus $\rho^{\perp}_0\equiv 0$, and the proof of Theorem \ref{main} is complete.
\end{proof}
For later use, we first introduce some basic notions. Let $f:M^n\rightarrow Q^{n+m}_c$ be an isometric immersion of a connected $n$-dimensional Riemannian manifold into an $(n+m)$-dimensional Riemannian manifold $Q^{n+m}_{c}$ of constant sectional curvature c, and let $\operatorname{II}:TM\times TM\rightarrow NM$
be the second fundamental form. For each $x\in M^{n}$, the space
\begin{equation*}
	N_{1}(x)={\rm Span}\{\operatorname{II}(X,Y):X,Y\in T_{x}M\}
\end{equation*}
 is called the first normal space of $f$ at $x$. An isometric immersion $f$ is said to be 1-regular if the first normal space $N_1(x)$ has constant dimension independent of $x\in M^n$. If $f$ is 1-regular, then $N_1(x)$ becomes a subbundle of $NM$, which is called the first normal bundle of $M^n$. 
   \begin{lem}[\cite{1984}]\label{1984}
	Let $f:M^n\rightarrow Q^{n+m}_{c}$ be a 1-regular immersion, where $Q^{n+m}_{c}$ denotes the simply connected, n-dimensional real space form with constant sectional curvature c. Write $$\operatorname{II}(e_i, e_j)=\sum_{\alpha=n+1}^{n+m}h_{ij}^\alpha e_\alpha\ \qquad
	\mbox{for}\ 1\leqslant i, j\leqslant n.$$ If for every point of $M$, there exists an orthonormal frame $\{E_1, \ldots, E_n\}$ such that for each $1\leqslant i\leqslant n$,
	$\operatorname{II}(E_i,E_j)=0$ for any $j\neq i$ and
	$\operatorname{II}(E_i,E_i)=\sum\limits_{j\neq i}a_j\operatorname{II}(E_j,E_j)$ for some real numbers $a_j$'s,
	then the first normal bundle $N_1$ is parallel.
\end{lem}

\begin{lem}[\cite{1971}]\label{1971}
	Let $f:M^n\rightarrow Q^{n+m}_{c}$ be an isometric immersion of a connected $n$-dimensional Riemannian manifold into an $(n+m)$-dimensional Riemannian manifold $Q^{n+m}_{c}$ of constant sectional curvature c. If the first normal bundle $N_1$ is parallel and dim $ N_1=p<m$, then the codimension of $f$ can be reduced to p.
\end{lem}
Applying Lemma \ref{1984} and Lemma \ref{1971}, we obtain the following pinching theorem.
   \begin{thm}\label{sec}
   		Let $M^n$ be an $n$-dimensional closed minimal submanifold immersed in the unit sphere $\mathbb{S}^{n+m}$ with  flat normal bundle. If $  \lambda_{1}\leqslant n$, then $M^n$ must be one of the following:
   		\begin{enumerate}
   			\item[\rm (1)]   the great sphere;
   			\item[\rm (2)]  
  the product $\prod\limits_{i=1}^{r+1}\mathbb{S}^{n_{i}}(\sqrt{\frac{n_{i}}{n}})$, where $\sum\limits_{i=1}^{r+1}n_{i}=n$, $1\leqslant r\leqslant m$. Moreover, $M^n$ satisfies the following properties: 
   (i) $\lambda_{1}=n$; (ii) $S=rn$.
    		\end{enumerate}
   \end{thm}
\begin{proof}
	Generally, we have Simons-type formula (\ref{Simons}):
\begin{equation*}
	\begin{aligned}
		\frac{1}{2}\Delta S=|\nabla h|^2+nS-\sum_{\alpha,\beta}|[A^{\alpha},A^{\beta}]|^2-\sum_{\alpha,\beta}|\langle A^{\alpha},A^{\beta}\rangle|^2.
	\end{aligned}
\end{equation*}
Since  the normal bundle of $M^n$ is flat, i.e., $\rho^{\perp}_0=\sum\limits_{\alpha,\beta}|[A^{\alpha},A^{\beta}]|^2=0$. 
Then integrating (\ref{Simons}) over $M$, we get
\begin{equation*}\label{inner 1}
	\begin{aligned}
		0\leqslant\int_{M}|\nabla h|^2dM&=\int_{M}-nS+\|\mathcal{A}\|^2dM\\
		&\leqslant\int_{M} -nS+S\lambda_{1}dM\\
		&=\int_{M}(\lambda_{1}-n)SdM.
	\end{aligned}
\end{equation*}
Therefore, if $0\leqslant\lambda_{1}\leqslant n$, then $S\equiv 0$ (just $\lambda_{1}\equiv 0$) or $\lambda_{1}\equiv n$. Meanwhile, $\forall i,j,k,\alpha$, $h_{ijk}^{\alpha}=0$ and there exists a positive integer $r$ such that 
\[\lambda_{1}=\lambda_{2}=\cdots=\lambda_{r} \quad {\text{and}} \quad  \lambda_{r+1}=\cdots=\lambda_{m}=0.\] 
There are two cases that need to be discussed:

Case (A).  For $\lambda_{1}\equiv 0$, we have $S\equiv 0$ and $M^n$ is the great sphere.

Case (B).  For  $\lambda_{1}\equiv n$,  one has $\lambda_{1}=\lambda_{2}=\cdots=\lambda_{r}=n$ and $\lambda_{r+1}=\cdots=\lambda_{m}=0$. Thus, $S\equiv rn$. The dimension of the first normal bundle $dim N_{1}=r$, which implies that the immersion is 1-regular. Since  the normal bundle of $M^n$ is flat, there exists a local orthonormal frame $\{e_{1}, e_{2},\ldots, e_{n}\}$, such that $h_{ij}^{\alpha}=\lambda_{i}^{\alpha}\delta_{ij}$ for every $\alpha$. Moreover, the functions $\lambda_{i}^{\alpha}$ are differentiable almost everywhere and globally continuous, and the local orthonormal frame exists and is differentiable almost everywhere. First, we have $h_{ij}^{\alpha}=0$ for $i\neq j$. Second, since $M^n$ is minimal, we get  $\sum\limits_{\alpha,i}h_{ii}^{\alpha}e_{\alpha}=0$, i.e., $\sum\limits_{\alpha}h_{ii}^{\alpha}e_{\alpha}=-\sum\limits_{\alpha}\sum\limits_{j\neq i}h_{jj}^{\alpha}e_{\alpha}$. Hence, by Lemma \ref{1984}, the first normal bundle $N_{1}$ is parallel. Applying Lemma \ref{1971}, the codimension of this immersion can be reduced to $r$. Specially, if $r=m$, then $M^n$ is full. Furthermore, since  the normal bundle of $M^n$ is flat, 
we can choose a local parallelizable othonormal normal frame so that $\omega_{\beta\alpha}=0$. Because $h_{ijk}^{\alpha}=0$ for all $i,j,k,\alpha$, then from the formula
\begin{equation}\label{cov}
	\sum_{k}h_{ijk}^{\alpha}\omega^{k}=dh_{ij}^\alpha+\sum_{l}h_{il}^{\alpha}\omega_{lj}+\sum_{l}h_{lj}^{\alpha}\omega_{li}+\sum_{\beta}h_{ij}^{\beta}\omega_{\beta\alpha},
\end{equation}
let $i=j$, we get
\begin{equation*}
	d\lambda_{i}^{\alpha}=0, \quad a.e.
\end{equation*}
Hence $\lambda_{i}^{\alpha}$ are constant almost everywhere. Then using formula (\ref{cov}) again, we conclude that 
\begin{equation*}
	(\lambda_{i}^{\alpha}-\lambda_{j}^{\alpha})\omega_{ij}=0.
\end{equation*} 
Therefore
\begin{equation}\label{key1}
	\omega_{ij}=0 \quad \text{whenever} \quad \lambda_{i}^{\alpha}\neq\lambda_{j}^{\alpha}.
\end{equation}
It follows that 
\begin{equation}\label{wedge}
	(\lambda_{i}^{\alpha}-\lambda_{j}^{\alpha})\omega_{ij}\wedge\omega_{j}=0,  \quad \forall i,j.
\end{equation} 
Without loss of generality, we assume that
\begin{equation*}
	\lambda_{1}^{\alpha}=\lambda_{2}^{\alpha}=\cdots=\lambda_{n_1}^{\alpha}>\lambda_{n_1+1}^{\alpha}=\cdots=\lambda_{n_2}^{\alpha}>\cdots>\lambda_{n_{k-1}+1}^{\alpha}=\cdots=\lambda_{n_k}^{\alpha}>\cdots=\lambda_{n}^{\alpha}.
\end{equation*}
Hence, combining with (\ref{wedge}), if $n_{k-1}<i\leqslant n_{k}$,
\begin{equation}\label{key2}
	d\omega_{i}=-\sum_{j=1}^{n}\omega_{ij}\wedge \omega_{j}=-\sum_{j=n_{k-1}+1}^{n_k}\omega_{ij}\wedge \omega_{j}.
\end{equation}
From (\ref{key1}) and (\ref{key2}) we see that $\omega_{j}=0$, $j\notin\{n_{k-1}+1,\ldots,n_{k}\}$ defines a totally geodesic foliation of $M^n$. Since all $\lambda_{i}^{\alpha}$ are constant, the leaves of this foliation are all closed and hence compact.
Varying $k$ yields different totally geodesic foliation with compact leaves. The leaves of these two foliations are mutually perpendicular to each other and equation (\ref{key1}) shows actually they give a decomposition of $M$. Thus $$M=M_{1}\times M_{2}\times\cdots\times M_{k}.$$
Since   the normal bundle of $M^n$ is flat, one has $h_{ij}^{\alpha}=\lambda_{i}^{\alpha}\delta_{ij}$ for every $\alpha$. Hence
\begin{equation}\label{Yau}
	h_{ij}^{\alpha}=0
\end{equation} for all $\alpha$ and $i\neq j$. Then following Yau's proof of Theorem 9 in \cite{Yau 1975}, considering the standard embedding of the sphere in Euclidean space, a lemma of J. D. Moore \cite{Moore}, and (\ref{Yau}) shows that each $M_{i}$ lies in a linear subspace $N_{i}$ and that the $N_{i}$'s are all mutually perpendicular, $k\leqslant r+1$. Similar to the proof of Theorem \ref{XW},  
 we can directly conclude that $$M=\mathbb{S}^{n_1}(\sqrt{\frac{n_1}{n}})\times \mathbb{S}^{n_2}(\sqrt{\frac{n_2}{n}})\times\cdots\times \mathbb{S}^{n_k}(\sqrt{\frac{n_k}{n}}),$$ where $\sum\limits_{i=1}^{k}n_{i}=n$. Since $S\equiv rn$, we have $k=r+1$. This completes the proof.
\end{proof}
\begin{proof}[\textbf{{Proof of Theorem \ref{thm S0n p=0}}}]
Combining Theorem \ref{main} with Theorem \ref{sec}, we complete the proof.
\end{proof}
\begin{proof}[\textbf{{Proof of Theorem \ref{spe}}}]
Since $S\leqslant n$, we have $\lambda_{1}\leqslant n$. Then by Theorem \ref{thm S0n p=0}, we conclude that either $S\equiv 0$ or $S=\lambda_{1}\equiv n$, which completes the proof of Theorem \ref{spe}.
\end{proof}
\begin{proof}[\textbf{Proof of Corollary \ref{constant}}]
	By (\ref{int}), if $\rho^{\perp}_0$ is constant, we obtain
	\begin{equation}\label{rho0const}
		\begin{aligned}
		0\leqslant&\int_{M}|\nabla h|^2+\big(\lambda_{1}+2S-n\big)\sqrt{\rho^{\perp}_0}dM\\
		=&\int_{M}\rho^{\perp}_0-nS+\|\mathcal{A}\|^2+\big(\lambda_{1}+2S-n\big)\sqrt{\rho^{\perp}_0}dM\\
		\leqslant&\int_{M}\rho^{\perp}_0-nS+S^2+(3S-n)\sqrt{\rho^{\perp}_0}dM\\
		=&\int_{M}\rho^{\perp}_0+(3S-n)\sqrt{\rho^{\perp}_0}+S(S-n)dM\\
		=&\rho^{\perp}_0\vol(M)+\sqrt{\rho^{\perp}_0}\int_{M}(3S-n)dM+\int_{M}S(S-n)dM.
		\end{aligned}
	\end{equation}
Since $\sqrt{\rho^{\perp}_0}$ is the solution of the following inequality in $x$:
\[x^2\vol(M)+x\int_{M}(3S-n)dM+\int_{M}S(S-n)dM\geqslant0,\]if $\delta=\left(\int_{M}(3S-n) dM\right)^2-4\vol(M)\int_{M}S(S-n)dM\geqslant 0$, then solving the quadratic inequality gives
	\begin{equation*}
	\sqrt{\rho^{\perp}_0}\leqslant\frac{\int_{M}(n-3S)dM-\sqrt{\delta}}{2\vol(M)} \quad  \text{or} \quad \sqrt{\rho^{\perp}_0}\geqslant\frac{\int_{M}(n-3S)dM+\sqrt{\delta}}{2\vol(M)}.
\end{equation*} 
By (\ref{rel}), we obtain
	\begin{equation*}
	\rho^{\perp}\leqslant\frac{\int_{M}(n-3S)dM-\sqrt{\delta}}{2n(n-1)\vol(M)} \quad \text{or} \quad \rho^{\perp}\geqslant\frac{\int_{M}(n-3S)dM+\sqrt{\delta}}{2n(n-1)\vol(M)}.
\end{equation*} 
\end{proof}
\begin{proof}[\textbf{Proof of Corollary \ref{Sconstant}}]
If $S$ is also constant, then by Corollary \ref{constant}, we obtain
\begin{equation*}
	\delta=\vol(M)^2\left((3S-n)^2-4S(S-n)\right)=\vol(M)^2(5S^2-2nS+n^2)>0.
\end{equation*}
By Corollary \ref{LL}, one has $S\geqslant\dfrac{2n}{3}$ if $S>0$.
Hence, we derive that
\begin{equation*}
		\rho^{\perp}\leqslant\frac{n-3S-\sqrt{5S^2-2nS+n^2}}{2n(n-1)}<0
 \quad \text{or} \quad
\rho^{\perp}\geqslant \frac{n-3S+\sqrt{5S^2-2nS+n^2}}{2n(n-1)}.
\end{equation*}
Since $\rho^{\perp}$ is nonnegative, the former case is impossible.
Notice that $\frac{n-3S+\sqrt{5S^2-2nS+n^2}}{2n(n-1)}\geqslant 0$ if and only if $0\leqslant S\leqslant n$. Hence, if $0\leqslant\rho^{\perp}\leqslant \frac{n-3S+\sqrt{5S^2-2nS+n^2}}{2n(n-1)}$, which in particular implies $0\leqslant S\leqslant n$, then $\rho^{\perp}=\frac{n-3S+\sqrt{5S^2-2nS+n^2}}{2n(n-1)}$. Consequently, the equalities hold in (\ref{rho0const}), yielding $$\lambda_{1}=S \quad \text{and} \quad|\nabla h|^2+\big(\lambda_{1}+2S-n\big)\sqrt{\rho^{\perp}_0}=0.$$ 
If $S\equiv 0$, then $M^n$ is the great sphere. If $S>0$, then by Corollary \ref{LL}, we have $S\geqslant \dfrac{2n}{3}$, thus $\lambda_{1}+2S-n=3S-n\geqslant n$, it follows that $\nabla h=0$ and $\rho_{0}^{\bot}\equiv 0$, which implies the normal bundle of $M^n$ is flat. Therefore, by Theorem \ref{spe}, $M^n$ must be either the great sphere or the Clifford torus.
\end{proof}

\section{Appendix}
In this section, we state the following theorem and provide a sketch of its proof, as given by Xu and Wang \cite{X-W}.
\begin{thm}[\cite{X-W}]\label{XW}
	Let $M^n$ be a connected smooth n-dimensional Riemannian manifold, it minimally immersed in a unit $(n+m)$-sphere $\mathbb{S}^{n+m}$. If the normal bundle of $M$ is flat and for each $\alpha$, the corresponding second fundamental form $A^{\alpha}$ is covariant constant over $M$, then $(M,\phi)$ is an open submanifold of one of the minimal products of spheres:$$\mathbb{S}^{n_1}(\sqrt{\frac{n_1}{n}})\times \mathbb{S}^{n_2}(\sqrt{\frac{n_2}{n}})\times\cdots\times \mathbb{S}^{n_k}(\sqrt{\frac{n_k}{n}})$$ with $\sum\limits_{i=1}^{k}n_{i}=n$.
\end{thm}
\begin{proof}[\textbf{Sketch of Proof}]
	Fix a point $q\in M$. Since the normal bundle of $M$ is flat, we can choose a local orthonormal base $\{e_{1}, e_{2},\ldots, e_{n}\}$ of $T_{q}M$ such that $h_{ij}^{\alpha}=\lambda_{i}^{\alpha}\delta_{ij}$ for every $\alpha$. Define an equivalence by $e_{i}\sim e_{j}$ if and only if $\lambda_{i}^{\alpha}=\lambda_{j}^{\alpha}$ for all $\alpha$. This divides $\{e_{1}, e_{2},\ldots, e_{n}\}$ into equivalence classes $E_{1},\ldots,E_{k}$. Denote by $n_{i}$ the number of the elements in $E_{i}$. Because $A^{\alpha}$ is covariant constant over $M$, then by paralleling moving $T_{q}M$ along a curve, we can obtain $k$ distributions $\{\Gamma_{i}\}$. Let $X\in \Gamma_{i}, Y \in T(M)$, for an arbitrary point $p\in M$, $\gamma:[0,1]\rightarrow M$ be a piecewise differential curve with $\gamma(0)=p,\gamma'(0)=Y$. By the definition of the covariant derivative $\nabla_{Y}X$, it follows that $\nabla_{Y}X\in \Gamma_{i}$. Consequently, if $X, Y\in \Gamma_{i}$, we have $[X,Y]=\nabla_{X}Y-\nabla_{Y}X\in \Gamma_{i}$. Thus, $\Gamma_{1},\ldots,\Gamma_{k}$ are involutive and the integral leaves of $\Gamma_{i}$ are totally geodesic submanifolds of $M$.
	
	Suppose $e_{l},e_{m}$ belong to two different classes, $E_{i}$ and $E_{j}$ respectively. From the proof of $\Gamma_{i}$'s involutive, we know $\langle R(e_{l},e_{m})e_{l},e_{m}\rangle=0$. By Gauss equation, this yields
	\[0=R_{lmlm}=1+\sum_{\alpha}\lambda_{l}^{\alpha}\lambda_{m}^{\alpha}.\]
	Denote by $B$ the $(m\times k)$-matrix
	\begin{equation*}
	B=\begin{pmatrix}
			\lambda_{1}^{n+1}&\lambda_{2}^{n+1}&\cdots&\lambda_{k}^{n+1}\\
			\lambda_{1}^{n+2}&\lambda_{2}^{n+2}&\cdots&\lambda_{k}^{n+2}\\
			\cdots&\cdots&\cdots&\cdots\\
			\lambda_{1}^{n+m}&\lambda_{2}^{n+m}&\cdots&\lambda_{k}^{n+m}\\
		\end{pmatrix}.
	\end{equation*}
Since $M$ is minimal, we have $\sum\limits_{i=1}^{k}\lambda_{i}^{\alpha}n_{i}=0$ for each $\alpha$, i.e.
\begin{equation*}
	B\begin{pmatrix}
		n_{1}\\\vdots\\ n_{k}
	\end{pmatrix}=0.
\end{equation*}
Then
\begin{equation*}
	B^{T}B\begin{pmatrix}
		n_{1}\\\vdots\\ n_{k}
	\end{pmatrix}=\begin{pmatrix}
		\sum\limits_{\alpha}(\lambda_{1}^{\alpha})^2&-1&\cdots&-1\\
		-1&\sum\limits_{\alpha}(\lambda_{2}^{\alpha})^2&\cdots&-1\\
		\cdots&\cdots&\cdots&\cdots\\
		-1&-1&\cdots&\sum\limits_{\alpha}(\lambda_{k}^{\alpha})^2\\
	\end{pmatrix}\begin{pmatrix}
	n_{1}\\\vdots\\ n_{k}
\end{pmatrix}=0.
\end{equation*}
A direct computation yields $\sum\limits_{\alpha}(\lambda_{i}^{\alpha})^2=\frac{n}{n_{i}}-1$, $i=1,\ldots,k$. Therefore, if $e_{m_{1}}, e_{m_{2}}$ both belong to $E_{i}$, we obtain
\[R_{m_{1}m_{2}m_{1}m_{2}}=1+\sum\limits_{\alpha}(\lambda_{i}^{\alpha})^2=\frac{n}{n_{i}}.\]
Thus, the integral manifold of $\Gamma_{i}$ has a constant sectional curvature $\dfrac{n}{n_{i}}$. Furthermore, by some discussion  in \cite{X-W}, there exists an isometric immersion:
\[\phi:M\rightarrow \mathbb{S}^{n_1}(\sqrt{\frac{n_1}{n}})\times \mathbb{S}^{n_2}(\sqrt{\frac{n_2}{n}})\times\cdots\times \mathbb{S}^{n_k}(\sqrt{\frac{n_k}{n}}).\]
\end{proof}
	
\end{document}